\documentclass[12pt,leqno,a4paper]{amsart}
\usepackage{amssymb, amsmath, enumerate,}
\usepackage{enumerate,color}
\usepackage[numbers]{natbib}
\usepackage{url}

\newtheorem{theorem}{Theorem}[section]
\newtheorem{lemma}[theorem]{Lemma}

\newtheorem*{thmA}{Theorem A}
\newtheorem*{thmB}{Theorem B}
\newtheorem*{thmC}{Theorem  C}
\newtheorem*{thmD}{Theorem  D}

\theoremstyle{remark}
\newtheorem{remark}[theorem]{Remark}

\newtheorem{definition}[theorem]{Definition}
\numberwithin{equation}{section}

\begin{document}
\title[Proper subgroups  of  order divisible  by $p$ are supersolvable]{Finite groups whose proper subgroups of order divisible by $p$ are supersolvable}
    \author[Beltr\'an and Viudez]{Antonio Beltr\'an\\
     Departamento de Matem\'aticas\\
      Universitat Jaume I \\
     12071 Castell\'on\\
      Spain\\
     \\Fernando Viudez \\
Departamento de Matem\'aticas\\
      Universitat Jaume I \\
     12071 Castell\'on\\
      Spain\\
     }

 \thanks{Antonio Beltr\'an: abeltran@uji.es ORCID ID: https://orcid.org/0000-0001-6570-201X \newline
 \indent Fernando Viudez: fviudez@uji.es }



\begin{abstract}
Let $p$ be an odd prime, and let $G$ be a finite  group  whose order is divisible by $p$. Suppose that every proper subgroup of $G$ whose  order is divisible by $p$ is supersolvable, while $G$ itself is not. In this paper, we investigate the arithmetic and structural properties of such groups, addressing both the solvable and nonsolvable cases.

\end{abstract}

\maketitle

\section{Introduction}

In this paper, all groups are assumed to be finite and standard notation is followed (e.g.
 \cite{KurzweilStellmacher}). Let $\mathcal{X}$ be a class of groups. A group  $G$ is called an $\mathcal{X}$-critical group if $G\notin \mathcal{X}$ and every proper subgroup of $G$ belongs to $\mathcal{X}$; that is, $G$ is a minimal non-$\mathcal{X}$-group. Well-known examples of critical groups include minimal non-abelian groups (Miller and Moreno \cite{MillerMoreno}), minimal non-nilpotent groups (Schmidt \cite{Schmidt}), and minimal non-supersolvable groups (Doerk \cite{Doerk}). Beyond these classical examples, $\mathcal{X}$-critical groups have  been studied in a  number of other classes of groups.
Following the nomenclature adopted by several authors \cite{LiMeng, Meng}, we introduce the following definition.

\begin{definition}
 Let $p$ be a prime number. We say that a group $G$ is a $pd$-group if its order is divisible by $p$. The group $G$ is called $\mathcal{X}$-semicritical with respect to $p$ if $G$ is a $pd$-group, $G\not\in \mathcal{X}$,  and every proper $pd$-subgroup of $G$ lies in $\mathcal{X}$.
 \end{definition}

  The study of $\mathcal{X}$-semicritical  groups with respect to a prime reveals minimal local obstructions that prevent a group from belonging to $\mathcal{X}$,  thus illustrating how local and global properties interact in the theory of classes of groups. However, as noted in \cite{Meng}, determining the structure of an $\mathcal{X}$-semicritical group is, in general, a difficult problem. A simpler situation arises, requiring only elementary results,  when $\mathcal{X}=\mathcal{N}$ or $\mathcal{A}$, that is, when considering the classes of nilpotent or abelian groups, respectively. In this context, Shi and Tian \cite{ShiTian}  recently provided a complete characterization of the $\mathcal{A}$-semicritical and $\mathcal{N}$-semicritical groups with respect to an arbitrary prime. Regarding $\mathcal{U}$, the class of supersolvable groups,   Lu and Meng \cite{Meng} were able to determine the solvability and several additional properties of $\mathcal{U}$-semicritical groups with respect to $2$. On the other hand, the case $p=3$ was studied by Shao and the first author in \cite{BeltranShao}, who based their approach on a refinement of the classification of minimal simple groups to identify nonabelian  simple  $\mathcal{U}$-semicritical groups  with respect to $3$. Another contribution, due to Li \cite{Li1988}, investigates nonsolvable groups whose second maximal $3d$-subgroups are  supersolvable.

\medskip
 In general, very little is known about the structure of $\mathcal{U}$-semicritical  groups with respect to odd primes other than $3$. This is precisely the objective  of  the present paper.
 Our first result reveals a highly restrictive structure in the nonsolvable case. Its proof is  based on the Classification of Finite Simple Groups, but unlike \cite{BeltranShao}, we rely on both elementary  arguments and deep results related to the coprime action on nonabelian simple groups. Recall that a group $G$ is said to be quasi-simple when $G$ is perfect (i.e. $[G,G]=G$) and $G/{\bf Z}(G)$ is a nonabelian simple group.

 \begin{thmA}
 Let $G$ be a  nonsolvable group and $p$ an odd prime. Then $G$ is $\mathcal{U}$-semicritical with respect to $p$ if and only if  $G$ is quasi-simple, ${\bf Z}(G)={\bf O}_{p'}(G)$, and $G/{\bf Z}(G)$ is   $\mathcal{U}$-semicritical with respect to $p$.
\end{thmA}

As mentioned above, [\citenum{BeltranShao}, Theorem 3.3] precisely determines which nonabelian simple groups  are $\mathcal{U}$-semicritical with respect to $p=3$, namely the groups ${\rm PSL}_2(2^r)$ for any odd prime $r$. However, establishing a complete classification of $\mathcal{U}$-semicritical  simple groups with respect to an arbitrary odd prime is a significantly more challenging task, especially for simple groups of Lie type; this task is not addressed in the present paper and remains open (see Remark \ref{remarkPSL} and Theorem \ref{theoremPSL}). Nevertheless,  we establish a criterion for alternating groups  to be $\mathcal{U}$-semicritical for an arbitrary prime and provide a complete list of those sporadic simple groups that are $\mathcal{U}$-semicritical  with respect to some prime. This is possible thanks to the recently completed classification of the maximal subgroups of the Monster group [\citenum{Dietrich}].

\begin{thmB}
The alternating group $A_n$ with $n\geq5$ is
$\mathcal{U}$-semicritical with respect to a prime $p$
if and only if either $n=p=5$ or $n=p>5$ and
$p\neq11,23$ and $p$ cannot be written in the form
$(q^d-1)/(q-1)$ for any prime power $q$ and any $d\geq2$.
\end{thmB}
 We remark that it is an open problem to determine  whether there exist finitely
or infinitely many primes $p$ that can represented by the arithmetic  expression appearing in  Theorem  B (see \cite{HardyWright}).

\begin{thmC}  The sporadic simple groups that are $\mathcal{U}$-semicritical with respect to some prime $p$ are in the following list.
\begin{itemize}
\item[(1)] $J_1$  is $\mathcal{U}$-semicritical with respect to $19$.
\item[(2)]  $M_{23}$ is $\mathcal{U}$-semicritical with respect to $23$.
\item[(3)]  $J_4$ is $\mathcal{U}$-semicritical with respect to $29$ and  $43$.
\item[(4)]  $Fi_{24}'$ is $\mathcal{U}$-semicritical with respect to $29$.
\item[(5)]  $Ly$ is $\mathcal{U}$-semicritical with respect to $37$ and  $67$.
\item[(6)]  $B$ is $\mathcal{U}$-semicritical with respect to $47$.
\item[(7)]  $M$ is $\mathcal{U}$-semicritical with respect to $59$.
\end{itemize}
\end{thmC}

In Section 4, we  examine the structure of solvable $\mathcal{U}$-semicritical groups with respect to an odd prime.  As mentioned  earlier, $\mathcal{U}$-semicritical groups with respect to $2$ are solvable and, in fact, are either $\mathcal{U}$-critical or $2$-nilpotent with Sylow $2$-subgroups of order $2$ [\citenum{Meng}, Theorem 1.2].  Note that supersolvable groups and $\mathcal U$-critical groups both have $p$-length $1$ for every prime divisor $p$ of their order. This  follows immediately  from the fact that every supersolvable group possesses a Sylow tower, and from the classification of minimal non-supersolvable groups (see, for instance, \cite{BallesterEsteban}), respectively. Theorem D  shows that solvable $\mathcal{U}$-semicritical groups with respect to an odd prime $p$ also  have $p$-length $1$. In fact, our hypotheses are more restrictive than those of [\citenum{Moreto}, Theorem C] since we focus on a single prime divisor of the group order. However, we are able to obtain further structural properties. We  denote by $\pi(G)$ the set of prime numbers that divide the order of a group $G$.

\begin{thmD}
 Let $p$ an odd prime. Suppose that $G$ is solvable and $\mathcal{U}$-semicritical with respect to $p$. Then $|\pi(G)|\leq 4$, the $p$-length of $G$  is $1$,  and one of the  following  possibilities occurs:
  \begin{itemize}
 \item[(1)]  $G$ is $\mathcal{U}$-critical (of order divisible by $p$);
 \item[(2)]  Assume ${\bf O}_{p}(G)>1$. Then  $G={\bf O}_p(G)M$, where ${\bf O}_p(G)$ is a minimal normal subgroup of $G$, and $M$ is a maximal $p'$-subgroup of $G$ that is  $\mathcal{U}$-critical. If, in addition, ${\bf O}_{p'}(G)=1$,  then $M$  is $\mathcal{A}$-critical,  $|G|_p\geq p^2$ and $|\pi(G)|=3$.
   \item[(3)] Assume  ${\bf O}_{p}(G)=1$. Then $|G|_p=p$ and $G=PM$, where $P$ is a Sylow $p$-subgroup of $G$ and $M$ is a maximal $p'$-subgroup of $G$ that is not supersolvable (not necessarily $\mathcal{U}$-critical).
 \end{itemize}

\end{thmD}

 In the  final section, we will  provide several examples to show that all the cases listed in Theorem D actually occur.

\section{Proof of Theorem A}

As mentioned in the Introduction, our strategy for proving Theorem A involves, among other things, some properties of coprime action on simple groups.   We  refer the reader to [\citenum{KurzweilStellmacher}, Chapter 8] for a detailed presentation of the standard results on coprime actions. We recall here only that if a  finite group $A$ acts coprimely on a finite group $G$, then, for every prime $p$ dividing the order of $G$, there exists an $A$-invariant Sylow $p$-subgroup in $G$. For convenience, we also state two additional results whose proofs are elementary.

\begin{lemma}\textnormal{[\citenum{KurzweilStellmacher}, 8.2.2]} \label{Lemma 1} Let $A$ be a group  acting  on a group $G$. Let $N$ be an $A$-invariant normal subgroup of $G$. Suppose that the action of $A$ on $N$ is coprime. Then ${\bf C}_{G/N}(A) = {\bf C}_G(A)N/N.$
\end{lemma}

\begin{lemma} %\textnormal{[\citenum{KurzweilStellmacher}, 8.1.6]}
\label{Lemma 2} Suppose that a finite group $A$ acts on a finite group $G$ that allows a direct decomposition $G = H_1 \times \cdots \times H_n$ such that $H_i^a \in\{H_1,\ldots, H_n\}$ for all  $ a \in A $ and $ i \in \{1,\ldots,n\}.$ Suppose further that $A$ acts transitively on $\{H_1,\ldots,H_n\}$. Let $H = H_1$ and $B = {\bf N}_A(H).$ Then ${\bf C}_G(A) \cong {\bf C}_H(B).$
\end{lemma}
\begin{proof}
    It follows  from [\citenum{KurzweilStellmacher}, 8.1.6]. 
\end{proof}

A much deeper result on coprime action, which depends on the Classification of Finite Simple Groups,   is required. Using the limited list of solvable Lie type  groups, one can derive the set of simple groups of Lie type that admit a nontrivial coprime automorphism group whose fixed-point subgroup is solvable.  The next theorem is well known and can be deduced, for instance,  from [\citenum{GuralnickShumyatsky}, Lemma 2.7 and Corollary 2.8].

\begin{theorem}\label{T2.3}
Let $G$ be a simple nonabelian group and $A$ an automorphism group of $G$ such that $(|G|,|A|)=1$. Then

\begin{itemize}
\item[(a)] If $G$ is alternating or sporadic, then $A=1$.
\item[(b)] If $G$ is a group of Lie type and ${\bf C}_G(A)$ is solvable, then $A$ is the full group of field automorphisms. Moreover, one of the following cases holds;

\begin{itemize}

\item[(1)] $G\cong {\rm PSL}_2(2^n)$, $n\geq 2$ and ${\bf C}_G(A)\cong \mathrm{PSL}_2(2)\cong S_3$;
\item[(2)] $G\cong {\rm PSL}_2(3^n)$, $n\geq 2$ and ${\bf C}_G(A)\cong \mathrm{PSL}_2(3)\cong A_4$;
\item[(3)] $G\cong {\rm PSU}_3(2^{n})$, $n\geq 2$ and ${\bf C}_G(A)\cong\mathrm{PSU}_3(2)\cong (C_3\times C_3)\rtimes Q_8$.;
\item[(4)] $G\cong \mathrm{Sz}(2^{2n+1})$, $n\geq 1$ and ${\bf C}_G(A)\cong \mathrm{Sz}(2)$, the Frobenius group of order $20$.
\end{itemize}
\end{itemize}
\end{theorem}

\medskip
We are ready to demonstrate Theorem A.

\begin{proof}[Proof of Theorem A]

We start with the direct implication. Let $S(G)$ denote the solvable radical of $G$ and set $\overline{G}=G/S(G)$. We first show that $S(G)\le {\bf O}_{p'}(G).$  Assume,  to the contrary, that $p$ divides $|S(G)|$.  Let $\overline{M}$ be a proper subgroup of $\overline{G}$ and let $M$ be its full preimage in $G$. Then $p$ divides $|M|$, so by hypothesis $M\in\mathcal U$, and therefore $\overline{M}\in\mathcal U$. Thus, every proper subgroup of $\overline{G}$ is supersolvable. Consequently, $\overline{G}$ is supersolvable or $\mathcal{U}$-critical; so in either case $\overline{G}$ is solvable, contradicting the non-solvability of $G$. Thus $S(G)\le {\bf O}_{p'}(G)$, as desired. 

\medskip
We now prove that $\Phi(G)=S(G)$. Obviously, $\Phi(G)\le S(G)$. Suppose,
for the sake of contradiction, that $\Phi(G)<S(G)$. Then there exists a maximal subgroup $M$ of $G$ such that $S(G)\nleq M$, so $G=MS(G)$. We know that $p$ does not divide $|S(G)|$, while $p\in \pi(G)$, so $p$ divides $|M|$. By the $\mathcal U$-semicritical assumption, $M\in\mathcal U$. Then the quotient $G/S(G)\cong M/(M\cap S(G))$ is supersolvable, against the non-solvability of $G$. Combining this with the result obtained earlier, we conclude $\Phi(G)=S(G)\le {\bf O}_{p'}(G).$ At this point, we distinguish two cases: whether or not the latter containment is proper.

\medskip
We  will first assume that $S(G)<{\bf O}_{p'}(G)$ and seek a contradiction. Choose a minimal normal subgroup $\overline{L}$ of $\overline{G}$ such that $\overline{L}\le \overline{{\bf O}_{p'}(G)}$ where $L$ is the full preimage of $\overline{L}$ in $G$. Then $S(G)<L\unlhd G$. We can write  $\overline{L}= S_1\times \cdots \times S_n$ for some $n\ge 1$, where $S_i$ is a nonabelian simple group for each $i$ and all factors are isomorphic. Choose an element $x\in G$ of order $p$, which necessarily lies outside $L$, and consider the subgroup $\langle x\rangle L$ of $G$. If $\langle x\rangle L<G$, then  the hypothesis would lead to $L \in\mathcal U$, a contradiction. Thus $\langle x\rangle L=G$. Now, the $p$-group $\langle x\rangle$ acts coprimely (via conjugation) on $L$ and on $\overline{L}$. In addition,  the fact that $\overline{L}$ is minimal normal in $\overline{G}$ implies that $\langle x\rangle$ acts transitively on $\{S_1,\dots,S_n\}$. Applying Lemma \ref{Lemma 2}, we have
 $${\bf C}_{\overline{L}}(x)\cong {\bf C}_{S_1}({\bf N}_{\langle x\rangle}(S_1)).$$
Consider the subgroup $\langle x\rangle {\bf C}_L(x).$ If this were equal to $G$, then $\langle x\rangle$ would be normal in $G$ (and solvable), which means $\langle x\rangle\le S(G)\le {\bf O}_{p'}(G)$, but this is not possible. Therefore, $\langle x\rangle {\bf C}_L(x)<G$. Moreover, since this subgroup has order divisible by $p$, then ${\bf C}_L(x)\in\mathcal U$. Passing to the quotient and taking into account Lemma \ref{Lemma 1}, we see that ${\bf C}_{\overline{L}}(x)$ is also supersolvable.  We  now prove that $n=1$, that is, $\overline{L}$ is simple. Suppose, on the contrary, that $n>1$. As the action of $\langle x\rangle$ is transitive and $x$ is of prime order, we certainly have ${\bf N}_{\langle x\rangle}(S_1)=1$. Hence ${\bf C}_{\overline{L}}(x)=S_1$, which is not solvable, a contradiction. This shows that $n=1$.  Now, since $\langle x\rangle$ has prime order, it acts faithfully on ${\overline{L}}$, and by Theorem \ref{T2.3}, we obtain that $\overline{L}$ is a simple group of Lie type. Moreover,  the only possibilities in which ${\bf C}_{\overline{L}}(x)$ is supersolvable are  ${\bf C}_{\overline{L}}(x)\cong {\rm PSL}_2(2)\cong S_3$ or ${\bf C}_{\overline{L}}(x)\cong \mathrm{Sz}(2)\cong C_5\rtimes C_4 $. Therefore,  again by Theorem \ref{T2.3}, we have $\overline{L} \cong {\rm PSL}_2(2^r)$  for some $r\geq 2$, or $\overline{L} \cong Sz(2^r)$ for some $r\geq 3$, respectively.

\medskip
By the elementary properties of coprime actions, we can take an $\langle x\rangle$-invariant  Sylow $2$-subgroup $Q$ of $L$.  Since $Q\langle x\rangle <G$, and  has order divisible by $p$, the hypothesis gives $Q\langle x\rangle\in \mathcal{U}$. In particular,  $Q\langle x\rangle$  possesses a normal $2$-complement, which means that $Q \leq {\bf C}_L(x)$. Then, using Lemma \ref{Lemma 1}, we have
$$\overline{Q} = \frac{Q\,S(G)}{S(G)} \le \frac{{\bf C}_L(x)\,S(G)}{S(G)} = {\bf C}_{\overline{L}}( x)$$
 where $\overline{Q}$ is a Sylow $2$-subgroup of $\overline{L}$. But  this leads to a contradiction. Indeed, since $|{\rm PSL}_2(2)|=6$ and
\[
|{\rm PSL}_2(2^r)| = \frac{2^r(2^{2r}-1)}{(2,2^r-1)} = 2^r(2^r-1)(2^r+1),
\]
we have that $|\overline{Q}| = 2^r$ must divide $2$, and this forces $r=1$, a contradiction. The same argument applies for ${\bf C}_{\overline{L}}(x) \cong \mathrm{Sz}(2)$ because $|\mathrm{Sz}(2)|=20$ and
\[
|\mathrm{Sz}(2^r)| = 2^{2r}(2^{2r}+1)(2^r-1).
\]
This implies that $|\overline{Q}| = 2^{2r}$ divides $4$, again a contradiction. This proves that  $S(G)<{\bf O}_{p'}(G)$ is impossible.

\medskip
Hence, from now on we assume $\Phi(G)=S(G)={\bf O}_{p'}(G)$.  Let $\overline{N}$ be a minimal normal subgroup of $\overline{G}$, where $N$ is its full preimage in $G$, so $S(G)<N\unlhd G$. This forces both $N$ and $\overline{N}$ to be nonsolvable. Moreover, since $S(G)={\bf O}_{p'}(G)$ and $p$ divides $|N|$, if $\overline{N}$ were a proper subgroup of $\overline{G}$, then  $N\in\mathcal U$, certainly a contradiction. Hence $\overline{N}=\overline{G}$, so $\overline{G}$ must be (nonabelian) simple. Let $\overline{M}$ be a proper subgroup of $\overline{G}$ with $p\in\pi(M)$, and let $M$ be its complete preimage in $G$. Then $M<G$ with $p$ dividing $|M|$, so by  hypothesis $M\in\mathcal U$, and consequently, $\overline{M}\in\mathcal U$. As a result, $\overline{G}$ is $\mathcal{U}$-semicritical with respect to $p$.

\medskip
We claim that $\Phi(G)={\bf Z}(G)$. By the simplicity of $G/\Phi(G)$, it is clear that $G$ cannot possess a normal $p$-complement, so, by Frobenius's  criterion for $p$-nilpotence  [\citenum{KurzweilStellmacher}, 7.2.4], we may ensure the existence of a nontrivial $p$-subgroup of $G$, say $P$,  such that ${\bf N}_G(P)/{\bf C}_G(P)$ is not a $p$-group.
This  means that there exists some prime $r\neq p$ that divides the order of ${\bf N}_G(P)/{\bf C}_G(P)$; equivalently, there exists an $r$-element $x\in {\bf N}_G(P)$ with $r\neq p$ such that $x\not \in {\bf C}_G(P)$. Consider the subgroup $P\langle x \rangle$ and observe that $P\langle x \rangle/P$ is abelian. Hence $P_0=[P\langle x \rangle,P\langle x \rangle ]$ is a nontrivial $p$-subgroup of $G$ since $x$ does not centralize $P$. Now, by hypothesis,  $P\langle x \rangle \Phi(G)$, which is a proper subgroup of $G$,  must be supersolvable. We then know that
$$[P\langle x \rangle \Phi(G),P\langle x \rangle \Phi(G)]  = P_0[\Phi(G), P\langle x \rangle][\Phi(G),\Phi(G)]$$ is nilpotent. In particular,  $P_0$ centralizes the $p'$-subgroup $[\Phi(G),P_0]$. Then, considering the coprime action of $P_0$ on $\Phi(G)$, we have $1=[\Phi(G),P_0,P_0]=[\Phi(G),P_0]$, which means  $1\neq P_0\leq  {\bf C}_G(\Phi(G))$. Now, since ${\bf C}_G(\Phi(G))\Phi(G)\unlhd G$ the simplicity of $G/\Phi(G)$ forces ${\bf C}_G(\Phi(G))\Phi(G)=G$. Hence ${\bf C}_G(\Phi(G))=G$, that is, ${\Phi}(G)\leq {\bf Z}(G)$, thus proving the claim. All that remains to be shown is that $G$ is perfect, which  follows immediately from the fact that $G/\Phi(G)$ is simple.

\medskip
To prove the converse, suppose that $G$ is quasi-simple with ${\bf O}_{p'}(G)={\bf Z}(G)$ and that $G/{\bf Z}(G)$  is $\mathcal{U}$-semicritical with respect to $p$. It is easy to show that ${\bf Z}(G)=\Phi(G)$. Indeed, if $M$ is a maximal subgroup of $G$ that does not contain ${\bf Z}(G)$, then $M{\bf Z}(G)=G$. Since $G$ is perfect, we have
$G=G' =(M{\bf Z}(G))' =M'$. This contradiction proves that ${\bf Z}(G)\leq \Phi(G)$, and then the simplicity of $G/{\bf Z}(G)$ yields equality. Now, let $H$ be a proper subgroup of $G$ of order divisible by $p$. The fact that ${\bf Z}(G)=\Phi(G)$ then implies  that $H{\bf Z}(G)/{\bf Z}(G)$ is a proper subgroup of $G/{\bf Z}(G)$, and its order is  also divisible by $p$. Therefore, the hypothesis gives us that
$$H{\bf Z}(G)/{\bf Z}(G)\cong H/H\cap{\bf Z}(G)$$ is supersolvable. Because $H\cap {\bf Z}(G)\leq {\bf Z}(H)$, it follows that $H/{\bf Z}(H)$ is supersolvable and, consequently, $H$ itself is supersolvable. This shows that $G$ is $\mathcal{U}$-semicritical with respect to $p$, as required.
\end{proof}

\begin{remark}
 Theorem A provides a complete characterization of finite nonsolvable  $\mathcal{U}$-semicritical groups with respect to any odd prime $p$, thereby strengthening  [\citenum{BeltranShao}, Theorem B] in the special case $p=3$. Notice further that the order of ${\bf Z}(G)$ is bounded by the order of the Schur multiplier of the chief factor $G/{\bf Z}(G)$. We also wish to acknowledge that T. Wilde has recently obtained a complete classification of the finite nonsolvable $\mathcal{U}$-semicritical groups with respect to $3$, which he kindly brought to our attention via private communication \cite{Wilde}.   
\end{remark}

\section{Proofs of Theorems B and C}

To determine which alternating groups are $\mathcal{U}$-semicritical with respect to a given prime, we will employ the well-known classification of transitive permutation groups of prime degree. Combining Galois's classification of solvable transitive groups of prime degree,  a result of Burnside \cite{Burnside1895}, and the Classification of Finite Simple Groups  (see \cite{Cameron1981} for example), we summarize these results as follows.

\begin{theorem}\label{T3.1} A transitive permutation group $G$ of prime degree $p$ is one of  the following:
\begin{itemize}
\item[(1)] In the solvable case, $G$ is isomorphic to a subgroup of  the general affine linear group ${\rm AGL}_1(p)$ that contains the translation subgroup (of order $p$).
\item[ (2)] In the nonsolvable case, $G$ is isomorphic to:
\begin{itemize}
 \item[(a)] $A_p$ or $S_p$ where $p\geq 5$;
 \item[(b)] a group $H$ such that  ${\rm PSL}_d(q)\leq H\leq {\rm P\Gamma L}_d(q)$ where $(q^d-1)/(q-1)=p$ with $q$ a prime power and $d\geq 2$.
\item[(c)] ${\rm PSL}_2(11)$ or $M_{11}$ when $p=11$, or $M_{23}$ when $p=23$.
\end{itemize}
\end{itemize}
\end{theorem}

We also recall the maximal subgroups of $\mathrm{PSL}_2(q)$. The following classification follows directly from
\cite[Tables 8.1 and 8.7]{BHR}.

\begin{theorem}\label{T2.4}
Let $q=p^e>3$ be a prime power. The maximal subgroups of
$\mathrm{PSL}_2(q)$ are given in the following list:
\begin{itemize}
    \item[(a)] dihedral groups of order $q-1$ for $q\geq13$ odd
    and $2(q-1)$ for $q$ even.

    \item[(b)] dihedral groups of order $q+1$ for $q\neq7,9$ odd
    and $2(q+1)$ for $q$ even.

    \item[(c)] the  Borel subgroups (the stabilizer of a point), of order
    $q(q-1)/2$ for $q$ odd and $q(q-1)$ for $q$ even.

    \item[(d)] $\mathrm{PSL}_2(q_0)$, where $q=q_0^r$,
    with $r$ an odd prime if $q$ is odd, and with $r$ prime
    and $q_0>2$ if $q$ is even.

    \item[(e)] $\mathrm{PGL}_2(q_0)$, where $q=q_0^2$ and
    $q$ is odd.

    \item[(f)] $S_4$ when $q\equiv\pm1\pmod8$,
    with $q$ prime.

    \item[(g)] $A_4$ when $q\equiv\pm3\pmod8$ and
    $q\not\equiv\pm1\pmod{10}$, with $q>3$ prime.

    \item[(h)] $A_5$ when $q\equiv\pm1\pmod{10}$,
    with either $q$ prime, or $q=p^2$ and
    $p\equiv\pm3\pmod{10}$.
\end{itemize}
\end{theorem}

We next record when a Borel subgroup of $\mathrm{PSL}_2(q)$ is supersolvable; this will be used subsequently.

\begin{lemma}\label{lem:borel}
Let $B$ be a point stabilizer in the natural action of
$\mathrm{PSL}_2(q)$ on the projective line, where
$q=\ell^f\geq4$ and $\ell$ is prime. Then $B$ is
supersolvable if and only if $f=1$.
\end{lemma}

\begin{proof}
Write $\delta=\gcd(2,q-1)$ and let $\mathbb F_q$ denote the field of $q$ elements. We use the standard structure for the Borel subgroup $B=U\rtimes T$, where $U\cong(\mathbb F_q,+)$ and $T\cong C_{(q-1)/\delta}$ acts on $U$ by field multiplication as $T$ is  the subgroup of squares in $\mathbb F_q^\times$. Under this action, it is well known that $U$ is a minimal normal subgroup of $B$. If $f=1$, then $B$ is clearly supersolvable, because $U$ and $T$ are cyclic groups. If $f>1$, however, the minimal normal subgroup $U$ of $B$ does not have prime order, so $B$ is not supersolvable.
\end{proof}

We now prove Theorem B using Theorem~\ref{T3.1}.

\begin{proof} [Proof of Theorem B] First, we analyze $A_5$. Since the maximal subgroups of $A_5$ are $D_{10}$, $A_4$ and $S_3$, we clearly see that $A_5$ is $\mathcal{U}$-semicritical only with respect to $5$. We may therefore assume for the rest of the proof that $n>5$. Also,  if $n$ is a composite number, then the set of primes dividing $n!$ is exactly the same as those dividing $(n-1)!$.  As $A_{n-1}\leq A_n$, we conclude  that $A_n$ cannot be $\mathcal{U}$-semicritical with respect to any prime that divides its order. Henceforth, we will assume that $n$ is prime, say $p$ (larger than $5$). In this case, for every prime divisor $r\neq p$ of the order of $A_p$, the same argument shows that $A_p$ is not $\mathcal{U}$-semicritical with respect to $r$. For that reason, we only need to study when $A_p$ is $\mathcal{U}$-semicritical with respect to $p$.

\medskip
First, since $M_{11}$ and $M_{23}$ are subgroups of $A_{11}$ and $A_{23}$, respectively, we deduce that both alternating groups cannot be $\mathcal{U}$-semicritical with respect to $11$ and $23$, respectively. Now, suppose that we can write $p=(q^d-1)/(q-1)$ for some prime power $q$ and some $d\geq 2$. It is known that ${\rm PSL}_d(q)$ acts in a natural and transitive way on the set of  $(q^d-1)/(q-1)=p$ points contained in the projective space ${\rm PG}(d-1,q)$.
 This means that ${\rm PSL}_d(q)\leq S_p$. Now, since this subgroup is nonabelian simple because $p>5$,  its simplicity certainly gives ${\rm PSL}_d(q)\leq A_p$. Furthermore, we claim that the equality does not hold. This is due to the fact that the only isomorphisms that exist between both families of simple groups are: ${\rm PSL}_2(4)\cong {\rm PSL}_2(5)\cong A_5$, ${\rm PSL}_2(9)\cong A_6$ and ${\rm PSL}_4(2)\cong A_8$; however, we are assuming $p>5$ and the pairs $(2,9)$ and $(4,2)$ do not satisfy the assumed arithmetic conditions. Thus, the fact that  the order of this proper subgroup of $A_p$ is divisible by $p$ implies that $A_p$ is not $\mathcal{U}$-semicritical with respect to $p$. This finishes the proof of the direct implication.

\medskip
Conversely, assume that $p\neq 11,23$ and that $p$ is not of the form
$(q^d-1)/(q-1)$ with $q$ a prime power and $d\geq 2$. We need to prove that every proper  subgroup $H$  of $A_p$ of order divisible by $p$ is supersolvable.
 It is clear that $H$ contains a $p$-cycle and acts transitively on a set of cardinality $p$. By applying Theorem \ref{T3.1}, taking into account that $H$ is proper in $A_p$ and our assumptions on $p$, the only possibility is that $H$ is  isomorphic to a subgroup of ${\rm AGL}_1(p)\cong C_p \rtimes C_{p-1}$, which is obviously supersolvable. Therefore, the converse implication is proved.
\end{proof}

\begin{proof}[Proof of Theorem  C]

To investigate the structure of the sporadic simple groups, we will primarily rely on  \cite{Atlas}. However, given that the description  of the maximal subgroups of  $Fi_{23}, Fi_{24}'$, $J_4$, $Th$, $B$ and $M$  is incomplete in \cite{Atlas}  and has been updated, we also consult more recent sources \cite{Dietrich, KleidmanJ4, KleidmanFi23,LintonTh, LintonThCorr, LintonFi24,WilsonB}. Throughout this proof, uniqueness statements about maximal
subgroups are understood up to conjugacy in the ambient group.

\medskip
First, we refer to the diagram presented on [\citenum{Atlas}, p. 238] which allows us to exclude certain sporadic groups from being $\mathcal{U}$-semicritical with respect to certain primes. Specifically, if any sporadic group contains another sporadic group as a factor of one of its subgroups, then the former cannot be $\mathcal{U}$-semicritical with respect to any prime divisor of the latter's order. Using this observation, we conclude that the only groups and primes that must be examined are the following:
 $M_{23}$ for $p=23$; $\mathrm{Sz}$ for $p=13$; $O'N$ for $p=31$; $Fi_{22}$ for $p=13$; $HN$ for $p=19$; $Ly$ for $p\in \{31,37,67\}$; $Fi_{23}$ for $p=17$; $J_{4}$ for $p\in\{29,31,37,43\}$; $Fi_{24}'$ for $p=29$; $B$ for $p=47$; and $M$ for $p\in\{41,59,71\}$. Likewise, $M_{11}$, $J_{1}$, $M_{22}$, $J_{2}$, $J_{3}$, $He$, $Ru$ and $Th$ need to be checked for all primes. By the  same observation, the groups $M_{12}$, $HS$, $M_{24}$, $McL$, $Co_{3}$, $Co_{2}$ and $Co_{1}$ can be straightforwardly discarded from being $\mathcal{U}$-semicritical with respect to any prime $p$.

\medskip
We also use \cite{Atlas} to search for non-supersolvable maximal subgroups whose prime divisors, taken collectively, cover all the primes dividing the order of certain sporadic groups, proving that those groups are not $\mathcal{U}$-semicritical with respect to any prime. Specifically, $M_{11}$ has maximal subgroups $A_6.C_2$ and ${\rm PSL}_2(11)$; $M_{22}$ has ${\rm PSL}_3(4)$ and ${\rm PSL}_2(11)$; $J_{2}$ has ${\rm PSU}_3(3)$ and $A_{5}$; $J_{3}$ has ${\rm PSL}_2(16) \rtimes C_{2}$ and ${\rm PSL}_2(19)$; $He$ has ${\rm Sp}_4(4) \rtimes C_2$ and $C_3. S_7$; $Ru$ has ${\rm PSL}_2(13) \rtimes C_2$ and ${\rm PSL}_2(29)$. Thus,  these groups can be excluded. In the case of $Th$, by appealing to \cite{LintonTh,LintonThCorr}, we also find three subgroups: ${\rm PSL}_3(3)$, ${\rm PSL}_2(19) \rtimes C_2$ and ${C_2^5 . \rm PSL}_5(2)$, whose orders include all the primes dividing its order.

\medskip
Continuing with \cite{Atlas}, to determine which of the remaining sporadic  groups  are not $\mathcal{U}$-semicritical with respect to any prime $p$, we exhibit examples of non-supersolvable maximal subgroups whose order is divisible by the remaining primes that need to be checked. In this way, for $p=13$, the group $\mathrm{Sz}$ contains $G_2(4)$; for $p=31$, $O'N$ has ${\rm PSL}_2(31)$; for $p=13$, $Fi_{22}$ contains the subgroup $O_7(3)$; for $p=19$, $HN$ has ${\rm PSU}_3(8)\rtimes C_3$. And finally, according to \cite{KleidmanFi23}, for $p=17$, the group $Fi_{23}$  contains $Sp_8(2)$. 

\medskip
To complete the proof, bearing in mind the analysis we have just carried out, we show exactly which sporadic groups are $\mathcal{U}$-semicritical with respect to a given prime, and thus justify the inclusion of each of the groups in the list of the statement. We start with $J_1$. The maximal subgroups ${\rm PSL}_2(11)$ and $C_2^3\rtimes (C_7\rtimes C_3)$ discard all primes except $19$, and the only maximal subgroup of order divisible by $19$ is $C_{19}\rtimes C_6$, which is supersolvable.  For  $M_{23}$, the only maximal subgroup of order  divisible by $23$ is $C_{23}\!\rtimes\!C_{11}$, which is supersolvable. The group $Ly$ contains $G_{2}(5)$ as a non-supersolvable subgroup of order divisible by $31$. However, the only maximal subgroup of $Ly$ that has order divisible by $37$ is $C_{37}\rtimes C_{18}$, and it is supersolvable. Likewise, the only maximal subgroup of $Ly$ of order divisible by $67$ is $C_{67}\rtimes C_{22}$, which is supersolvable too. According to \cite{KleidmanJ4}, $J_{4}$ contains $(C_{2})^{10}\rtimes {\rm 
PSL}_{5}(2)$ and ${\rm PSU}_{3}(11) \rtimes C_{2}$,  of orders divisible by $31$ and $37$, respectively.  Nonetheless, the only maximal subgroups of $J_{4}$ of order  divisible by $29$ and $43$ are $C_{29}\rtimes C_{28}$ and $C_{43}\rtimes C_{14}$, respectively, and both are supersolvable. On the other hand, taking \cite{LintonFi24} into account, the only maximal subgroup of $Fi_{24}'$ of order divisible by $29$ is $C_{29}\rtimes C_{14}$, which is supersolvable. Regarding $B$,  using \cite{WilsonB}, we know that the only maximal subgroup of order divisible by $47$ is $C_{47}\rtimes C_{23}$, which is supersolvable. Finally, according to \cite{Dietrich}, $M$ contains ${\rm PSL}_2(71)$ and ${\rm PSL}_{2}(41)$, so $71$ and $41$ can be discarded. However,  $M$ has a unique maximal subgroup of order divisible by $59$, namely, $C_{59}\rtimes C_{29}$, which is supersolvable. This completes the list in the statement.
\end{proof}

\begin{remark}\label{remarkPSL} As noted in the Introduction, determining which simple groups of Lie type are $\mathcal{U}$-semicritical with respect to a given prime  appears significantly more intricate, owing to the complexity of their maximal subgroups   and distinct arithmetical relations among the primes involved. Addressing this is beyond the present paper. However, due to the simplicity of the subgroup lattice of ${\rm PSL}_2(q)$, we  include a  complete classification of $\mathcal{U}$-semicritical  groups with respect to an arbitrary prime for this family.
\end{remark}

\begin{theorem}\label{theoremPSL}
Let $q=\ell^f\geq4$, where $\ell$ is prime, and let $p$
be a prime dividing $|\mathrm{PSL}_2(q)|$. Then
$\mathrm{PSL}_2(q)$ is $\mathcal{U}$-semicritical with
respect to $p$ if and only if one of the following holds:
\begin{itemize}
\item[(1)] $p=3$, $\ell=2$, and $f$ is an odd prime;
\item[(2)] $p=5$ and $q\in\{4,5\}$;
\item[(3)] $p\geq7$ and either $f=1$, or $f>1$ with
$\operatorname{ord}_p(\ell)=2f$ $(\operatorname{ord}_p(\ell)$ denotes the multiplicative order of $\ell$ modulo $p$, for $p\neq \ell)$.
\end{itemize}

\end{theorem}

\begin{proof}
Set $G=\mathrm{PSL}_2(q)$ and $\delta=\gcd(2,q-1)$. Recall that $|G|=q(q-1)(q+1)/\delta$. Since $q\geq4$, the group $G$ is nonabelian simple and, in particular, is not supersolvable. We may assume that $p\geq3$, since, as noted in the Introduction, $\mathcal U$-semicritical groups with respect to $2$ are solvable [\citenum{Meng}, Theorem 1.2]. We also use the natural embedding $\mathrm{PSL}_2(\ell^a)\leq\mathrm{PSL}_2(\ell^f)$ whenever $a\mid f$. We distinguish three cases.

\emph{Case $p=3$.}
Suppose first that $q$ is odd. If $f>1$, the proper subgroup $H\cong\mathrm{PSL}_2(\ell)$ has order divisible by $3$ and is not supersolvable: it is isomorphic to $A_4$ if $\ell=3$, and is nonabelian simple otherwise. If $f=1$, Theorem~\ref{T2.4}(f)--(h) provides a maximal subgroup isomorphic to $A_4$, $S_4$ or $A_5$, each of which has order divisible by $3$ and is not supersolvable. Hence $q$ must be even.

Write $q=2^f$. If $f$ is composite, choose a proper divisor $a$ of $f$ (so $1<a<f$). The proper subgroup $H\cong\mathrm{PSL}_2(2^a)$ is nonabelian simple and has order divisible by $3$, a contradiction. If $f=2$, then $G\cong A_5$, and this case is excluded by Theorem~B. Thus $f$ is an odd prime. Conversely, let $q=2^f$ with $f$ an odd prime. Since $\mathbb F_2$ is the only proper subfield of $\mathbb F_q$ and $3\nmid q(q-1)$, Theorem~\ref{T2.4} shows that every maximal subgroup whose order is divisible by $3$ is dihedral and, therefore, supersolvable.

\emph{Case $p=5$.}
If $q=4$ or $5$, then $G\cong A_5$, and the assertion follows from Theorem~B. Suppose that $q>5$. We show that $G$ contains a proper subgroup isomorphic to $A_5$, which is non-supersolvable and of order divisible by $5$. If $\ell=2$, then the fact that $5$ divides $|G|$ implies $1\equiv2^{2f}\equiv(-1)^f\pmod5$, so $f$ is even and $\mathrm{PSL}_2(4)\cong A_5$ embeds in $G$. Suppose that $\ell$ is odd. If $\ell=5$, then $\mathrm{PSL}_2(5)\cong A_5$ embeds in $G$. If $\ell\equiv\pm1\pmod5$, then $\ell\equiv\pm1\pmod{10}$, and Theorem~\ref{T2.4}(h) provides a subgroup $A_5$ of $\mathrm{PSL}_2(\ell)\leq G$. Finally, if $\ell\equiv\pm2\pmod5$, then the fact that $5$ divides $|G|$ gives $1\equiv\ell^{2f}\equiv(-1)^f\pmod5$, so $f$ is even. Hence $\mathrm{PSL}_2(\ell^2)\leq G$; since $\ell$ is odd, $\ell\equiv\pm3\pmod{10}$, and Theorem~\ref{T2.4}(h) provides a subgroup $A_5$ of $\mathrm{PSL}_2(\ell^2)$.

\emph{Case $p\geq7$.}
If $f=1$, Theorem~\ref{T2.4} shows that no maximal subgroup is of type (d) or (e), and the exceptional subgroups $A_4,S_4,A_5$ have orders not divisible by $p$. Thus every maximal subgroup whose order is divisible by $p$ is Borel or dihedral, and both are supersolvable (appealing to Lemma~\ref{lem:borel} for the former). Suppose now that $f>1$ and that $G$ is $\mathcal U$-semicritical with respect to $p$. The proper Borel subgroup $B$ of $G$ is not supersolvable by Lemma~\ref{lem:borel}, so $p\nmid|B|=q(q-1)/\delta$. Thus $p\mid q+1$, so $\operatorname{ord}_p(q)=2$. Let $e=\operatorname{ord}_p(\ell)$ and $d=\gcd(e,f)$. Then
\[
  2=\operatorname{ord}_p(\ell^f)
  =\frac{e}{\gcd(e,f)}=\frac ed.
\]
Therefore, $e=2d$ and $d\mid f$. If $d<f$, the proper subgroup $K\cong\mathrm{PSL}_2(\ell^d)$ has order divisible by $p$, since $\ell^d\equiv-1\pmod p$. Since $\ell^d\geq p-1\geq6$, the group $K$ is nonabelian simple, a contradiction. Hence $d=f$ and $\operatorname{ord}_p(\ell)=2f$. Conversely, suppose that $f>1$ and $\operatorname{ord}_p(\ell)=2f$. Then $p\mid q+1$ and $p\nmid q(q-1)$, so the maximal Borel subgroups have orders not divisible by $p$, as do $A_4,S_4,A_5$. By Theorem~\ref{T2.4}, every remaining maximal subgroup that is not dihedral is isomorphic to $\mathrm{PSL}_2(\ell^a)$ or $\mathrm{PGL}_2(\ell^a)$ for some proper divisor $a$ of $f$. The orders of these groups divide $\ell^a(\ell^{2a}-1)$, which is not divisible by $p$: otherwise, $2f=\operatorname{ord}_p(\ell)\mid2a$, contrary to $a<f$. Thus every maximal subgroup whose order is divisible by $p$ is dihedral and, therefore, supersolvable.
\end{proof}

\section{Solvable $\mathcal {U}$-semicritical groups}

To deal with solvable $\mathcal{U}$-semicritical groups, we will  utilize standard properties concerning solvable, supersolvable and minimal non-supersolvable groups,  as well as the classification of minimal nonabelian groups (see, for instance, [\citenum{BallesterEsteban}, Theorem 1]). In accordance with the latter classification, one easily deduces the following.

\begin{lemma}\label{Lemma 4.1}  A minimal nonabelian group that is not supersolvable has the following structure: $G=V_q \rtimes C_{r^s}$, where $q$ and $r$ are different prime numbers, $s$ is a positive
integer, and $V_q$ is an irreducible $C_{r^s}$-module over the field of $q$ elements of dimension at least $2$,
with the kernel being the maximal subgroup of $C_{r^s}$.
\end{lemma}

We will also make  use of the next result to prove Theorem D.

\begin{lemma}\label{Lemma 4.2}\textnormal{ \cite{Doerk}} Let $G$ be a finite group. Suppose that $A$, $B$, $C$ and $D$ are supersolvable subgroups of $G$. If any pairs of $\lvert G:A\rvert$, $\lvert G:B\rvert$, $\lvert G:C\rvert$ and $\lvert G:D\rvert$ are coprime, then $G$ is supersolvable.
\end{lemma}

\begin{proof} [Proof of Theorem D]
We begin by showing that $|\pi(G)|\leq 4$. 
Set $\pi(G)=\{p_0, p_1, \ldots, p_r\}$ with $p_0=p$. Since $G$ is solvable, it possesses a Hall $p_i'$-subgroup, say $H_i$, for every prime $p_i$ with $i=0,1,\ldots r$. Suppose that $|\pi(G)|>4$, so $r > 3$.  Then, for every $i=1, \ldots, r$, we have that $H_i$ is a proper subgroup of $G$ of order divisible by $p$, so by hypothesis, every $H_i$ is supersolvable. Thus, there are at least four subgroups $H_{1},H_{2},H_{3}, H_{4}$ whose indices in $G$ are powers of distinct primes. In particular, the indices are pairwise coprime. Applying Lemma \ref{Lemma 4.2}, we would get that $G$ is supersolvable, contradicting our assumptions. This proves $|\pi(G)|\le 4$.

\medskip
Next, we prove that $G$ satisfies  the properties appearing in  each of the cases of the statement and that the $p$-length of $G$ is $1$. If $G$ is $\mathcal{U}$-critical (of order divisible by $p$), it is clear that $G$ has $p$-length $1$ according to the structure of such groups [\citenum{BallesterEsteban}, Theorem 10]. Henceforth, for the rest of the proof we will assume that $G$ is not $\mathcal{U}$-critical, that is,  there exists a maximal subgroup $M$ of $G$ that is not supersolvable. By hypothesis, $M$ must be a $p'$-group, and then by the solvability of $G$, it is clear that $M$ is a Hall $p'$-subgroup of $G$.

\medskip
We first consider the case  ${\bf O}_p(G)\neq 1$ and prove (2).  The maximality of $M$ gives  $G=M{\bf O}_p(G)$, and obviously, $G$ has $p$-length 1. Suppose that there exists a minimal normal subgroup $L$ of $G$ properly contained in ${\bf O}_p(G)$. Then $LM <G$ is supersolvable by hypothesis, contradicting the non-supersolvability of $M$. Accordingly,  ${\bf O}_p(G)$ must be a minimal normal subgroup of $G$. Furthermore, if $K$ is a maximal subgroup of $M$, then ${\bf O}_p(G)K<G$ is supersolvable, so $K$ is also supersolvable, and this proves that $M$ is $\mathcal{U}$-critical. Thus, the first part of (2) is proved.
Suppose, in addition, that  ${\bf O}_{p'}(G)=1$. Then, for every maximal subgroup $K$ of $M$, the subgroup $K{\bf O}_p(G)$ is supersolvable by hypothesis and, as a consequence, $(K{\bf O}_p(G))'=K'[K, {\bf O}_p(G)]$ is nilpotent. It follows that  $K'$, being a normal $p$-complement of $(K{\bf O}_p(G))'$, centralizes $[K,{\bf O}_p(G)]$, so in particular, $[{\bf O}_p(G),K',K']= 1$.  By  coprime action, we deduce that $K'\subseteq {\bf C}_G({\bf O}_{p}(G))$. Now, since ${\bf O}_{p'}(G)=1$, it is known that $ {\bf C}_G({\bf O}_{p}(G))\subseteq {\bf O}_p(G)$, which gives $K'=1$. Moreover, as seen above, $M$ is not supersolvable,  in particular, is not abelian, so we conclude that $M$ is also $\mathcal{A}$-critical. Furthermore,  in view of [\citenum{ShiTian}, Theorem 1.1], we have $|\pi(M)|=2$, and hence $|\pi(G)|=3$. Finally, observe that the action of $M$ on ${\bf O}_p(G)$ is faithful because ${\bf O}_{p'}(G)=1$, so $M\leq$ Aut$({\bf O}_p(G))$. We deduce that $|{\bf O}_p(G)|\geq p^2$; otherwise, $M$ would be abelian, a contradiction. Therefore, the second part of (2) is proved.

\medskip
To prove (3),  we assume ${\bf O}_p(G)=1$, and hence, by solvability,  ${\bf O}_{p'}(G)>1$. Let $H={\bf O}_{p',p}(G)$ (possibly equal to $G$) and $P$ a Sylow $p$-subgroup of $G$. The maximality of $M$ indicates that $MH=G$, and hence $P\leq H$. In particular, $G$ has $p$-length $1$. On the other hand, note that ${\bf O}_{p'}(G)\leq M$, and we distinguish two possibilities. If  ${\bf O}_{p'}(G)= M$,  then the normality of $M$ clearly yields $|P|=p$, so we get all the conditions in (3). Henceforth, we assume that
${\bf O}_{p'}(G)< M$.  Consider ${\bf N}_G(P)$, which is a proper subgroup of $G$, so it is supersolvable by hypothesis. Take $P_0$ to be a minimal normal subgroup of ${\bf N}_G(P)$ with $P_0\leq P$. Then, the supersolvability of ${\bf N}_G(P)$ implies that $|P_0|=p$. Also, the Frattini argument gives $G=H{\bf N}_G(P)$, and by Dedekind's rule $M={\bf O}_{p'}(G){\bf N}_M(P)$. Now, since $P_0{\bf N}_M(P)\leq G$, then $P_0{\bf N}_M(P){\bf O}_{p'}(G)=P_0M$ is also a subgroup of $G$. By maximality, $P_0M=G$, and this forces $P_0=P$. Thus, we again get the properties described in (3), so the proof is finished.
\end{proof}

\begin{remark}\label{Remark 4.3Ex} In  Section 5, we will see that, in case (2) of Theorem D, the subgroup ${\bf O}_p(G)$ may not have order $p$ when ${\bf O}_{p'}(G)>1$. Moreover, we will show that, unlike case (2),  the maximal subgroup $M$ in case (3) does not  need to be $\mathcal{U}$-critical.
 \end{remark}

\section{Examples}

We illustrate examples of groups  satisfying each of the conditions obtained in the  distinct cases of Theorem D. When possible, we will refer to the {\sf SmallGroup } library of GAP \cite{GAP4}, to facilitate the reader's verification.

\medskip
\noindent
{\bf Example 1}. To obtain examples of case (1), it suffices to take any $\mathcal{U}$-critical group of order divisible by $p$, which may belong to  any of the 11 distinct types of groups described in [\citenum{BallesterEsteban}, Theorem 10].

\medskip
\noindent
{\bf Example 2}. 
 In case (2), one may take any $\mathcal{U}$-critical  group $H$ of $p'$-order, and then construct the direct product $G=C_p\times H$, which certainly satisfies the required conditions. 
 
 \medskip
 \noindent
 {\bf Example 3}.
Another  example of case (2) of Theorem D in which ${\bf O}_p(G)$ is non-central in $G$ is given by $$G={\sf SmallGroup}(84,11)\cong C_7 \rtimes A_4,$$ where $A_4$ acts on $C_7$ with $C_2\times C_2$ contained in its kernel.  We have that $G$ is $\mathcal{U}$-semicritical with respect to $7$,  ${\bf  O}_7(G)\cong C_7$ and $M=A_4$ is a maximal $\mathcal{U}$-critical $7'$-subgroup. Also, notice that $1\neq {\bf O}_{p'}(G)\cong C_2\times C_2 $.

\medskip
\noindent
{\bf Example 4}. This example shows that in case (2) we cannot guarantee that $|{\bf O}_p(G)|=p$ when ${\bf O}_{p'}(G)>1$. We start from the group  $H\cong {\rm SL}_2(3)$, which admits the presentation
$$H=\langle a, b, t \mid a^4=b^4=t^3=1, a^2=b^2, a^b= a^{-1}, a^t=b, b^t =ab\rangle . $$
Consider the action of $H$ on the vector space $V=\mathbb{F}_7^3$,  given by the following matrices with entries in $\mathbb{F}_7$:

$$ a= \begin{pmatrix}
1~~~ & 0 & ~~~0 \\
0~~~ & -1~~~ & ~~~0\\
0~~~ & 0 &-1
\end{pmatrix},
\, b =
\begin{pmatrix}
-1            & \phantom{-}0 & \phantom{-}0 \\
\phantom{-}0  & \phantom{-}1 & \phantom{-}0 \\
\phantom{-}0  & \phantom{-}0 & -1
\end{pmatrix}
,
\, t =
\left(
\begin{array}{rrr}
0 & 1 & 0 \\
0 & 0 & 1 \\
1 & 0 & 0
\end{array}
\right)
.$$
Notice that ${\bf Z}(H)\cong C_2$ lies in the kernel of the action. Then, it is easy to check that the group $G=V\rtimes H$ is $\mathcal{U}$-semicritical with respect to $7$. Furthermore, ${\bf O}_7(G)=V$ and ${\bf O}_{7'}(G)= {\bf Z}(H)$.

\medskip
\noindent
 {\bf Example 5}.
  The following provides an example of case (2) of Theorem D with ${\bf O}_{p'}(G)=1$. Let $K= {\sf SmallGroup} (75, 2) = \langle a, b , r\rangle\cong (C_5\times C_5)\rtimes C_3$,  acting faithfully on the vector space $V=\mathbb{ F}_{31}^3$.    The action is described by the following matrices with entries in $\mathbb{F}_{31}$:

  $$ a= \begin{pmatrix}
1 & 0 & 0 \\
0 & 2 & 0\\
0 & 0 & 16
\end{pmatrix},
\, b= \begin{pmatrix}
2 & 0 & 0 \\
0 & 1 & 0\\
0 & 0 & 16
\end{pmatrix},
\, r= \begin{pmatrix}
0 & 1 & 0 \\
0 & 0 & 1\\
1 & 0 & 0
\end{pmatrix}
.$$
The group  $G=V\rtimes K$ is $\mathcal{U}$-semicritical with respect to $p=31$, with $K$  being a maximal subgroup of $p'$-order that is $\mathcal{U}$-critical, and $\mathcal{A}$-critical. Also, note that $|{\bf O}_p(G)|=p^3$.

\medskip

 We  emphasize that the  maximal  non-supersolvable subgroup $M$  arising in the proof of case (3) in Theorem D may fall into one of two possibilities: ${\bf O}_{p'}(G)<M$ or ${\bf O}_{p'}(G)=M$. We illustrate these two cases with the next two examples.  

\medskip
\noindent
{\bf Example 6}.
For the case ${\bf O}_{p'}(G)<M$, consider $H\cong C_7\rtimes C_3$, the non-abelian group of order $21$, which acts on the $3$-dimensional vector space $V=\mathbb{ F}_{29}^3$ over the field of $29$ elements (indeed, $H\leq {\rm GL}_3( 29)$). The action can be  described  by the following matrices, whose entries lie in $\mathbb{F}_{29}$:
$$ a= \begin{pmatrix}
16 & 0 & 0 \\
0 & 24 & 0\\
0 & 0 & 25
\end{pmatrix},
\, b= \begin{pmatrix}
0 & 1 & 0 \\
0 & 0 & 1\\
1 & 0 & 0
\end{pmatrix}.
$$
These matrices satisfy $a^7=1, b^3=1$ and $a^b=a^4$, whence $H=\langle a, b  \rangle$ is isomorphic to $C_7\rtimes C_3$. The corresponding  semidirect product  $G=V\rtimes H$ is then   $\mathcal{U}$-semicritical with respect to $7$. In addition, $M=V\rtimes \langle b \rangle$ is a maximal subgroup of $7'$-order that properly contains ${\bf O}_{7'}(G)=V$. Furthermore, in response to Remark \ref{Remark 4.3Ex}, we can conclude that $M$ is not $\mathcal{U}$-critical. Indeed, the action of $C_3$ on $V$ gives rise to precisely one invariant subspace of dimension $1$,  say $V_1$, and one invariant subspace $V_2$ of dimension 2, in such a way that $V=V_1\times V_2$ is the only decomposition of $V$ into invariant subspaces (in fact, the action is semisimple by Masche's Theorem   [\citenum{KurzweilStellmacher}, 8.4.6]).  Consequently, $M$ is not supersolvable. Moreover, since $V_2\rtimes \langle b \rangle$ is not supersolvable either,  $M$ is not $\mathcal{U}$-critical.

\medskip
\noindent
{\bf Example 7}.
For the second possibility, that is, ${\bf O}_{p'}(G)=M$, we may take 
$$G={\sf SmallGroup}(588, 34)\cong (C_7\times C_7)\rtimes C_{12},$$ which has the following presentation:
$$G=\langle x, y, t \mid  x^7=y^7=t^{12}=1, [x,y]=1, x^t=xy^2, y^t=xy^6\rangle .$$
This group is $\mathcal{U}$-semicritical with respect to $3$ and  satisfies ${\bf O}_{3'}(G)\cong (C_7\times C_7)\rtimes C_4$, which is a maximal $3'$-subgroup that is not supersolvable because $\langle t^3\rangle \cong C_4$ acts irreducibly on $\langle x \rangle \times \langle y  \rangle \cong C_7\times C_7$.

\bigskip
\noindent
{\bf Acknowledgements}
The results of this paper are part of the second author's Ph.D. thesis at the Universitat Jaume
I of Castell\'on. The first author is partially supported by the National Nature Science Fund (No. 12071181) of the People's Republic of China. We thank the anonymous referee for insightful suggestions regarding the Lie group approach, which we plan to address in future work.

\medskip
\noindent
{\bf Data availability} Data sharing is not applicable to this article as no data sets were generated or analyzed during
the current study.

\bigskip
\noindent
{\bf \large Declarations}

\medskip
\noindent
{\bf Conflict of interest} The authors have no conflicts of interest to declare.

\bibliographystyle{plain}

\begin{thebibliography}{1}



\bibitem{BallesterEsteban}
Ballester-Bolinches, A., Esteban-Romero, R.:
On minimal non-supersoluble groups.
\newblock {\em Rev. Mat. Iberoam.} {23}, no. 1, 127--142 (2007).

\bibitem{BeltranShao}
Beltrán, A., Shao, C.:
Supersolvable subgroups of order divisible by 3.
{\em Results Math.} {80}, no. 5, 127 (2025). https://doi.org/10.1007/s00025-025-02443-0


\bibitem{BHR}
Bray, J.,  Holt, D., Roney-Dougal, C.M.: The Maximal Subgroups of the Low-Dimensional Finite Classical Groups, London Math. Soc. Lecture Note Series, vol. 407. Cambridge University Press, Cambridge (2013).

\bibitem{Burnside1895}
Burnside, W.:
Notes on the theory of groups of finite order
\newblock{\em Proc. London Math. Soc.} 26, no. 1, 191–214 (1895).
%Collected Papers, vol. I, Note VII, 561–214. 
https://doi.org/10.1112/plms/s1-26.1.191

\bibitem{Cameron1981}
Cameron, P. J.:
Finite permutation groups and finite simple groups
\newblock{\em Bull. London Math. Soc.} 13. 1–22 (1981).

\bibitem{Atlas}
Conway, J. H., Curtis, R. T., Norton, S. P., Parker, R. A., Wilson, R. A.:
Atlas of finite groups.
\newblock Oxford University Press, Oxford (1985).

\bibitem{Dietrich}
Dietrich, H., Lee, M., Popiel, M.: The maximal subgroups of the Monster. 
\newblock{\em Adv. Math.} { 469}, 110214 (2025).
https://doi.org/10.1016/j.aim.2025.110214

\bibitem{Doerk}
Doerk, K.: Minimal nicht \"uberaufl\"osbare, endliche Gruppen. Math. Z. { 91}, no. 3, 198-205 (1966).

%\bibitem{Doerk2}
%Doerk, K., Klaus, T.: Finite soluble groups. 
%De Gruyter Exp. Math., 4  
%Walter de Gruyter, no. 4, Berlin (1992). 

\bibitem{GAP4}
GAP -- Groups, Algorithms, and Programming, Version 4.14.0
\newblock The  GAP  Group (2024).
\newblock \url{https://www.gap-system.org}.

\bibitem{GuralnickShumyatsky}
Guralnick, R., Shumyatsky, P.:
Derived subgroups of fixed points.
\newblock{\em Israel J. Math.} 126, no.1, 345-362 (2001).

\bibitem{HardyWright}
Hardy, G.H., Wright, E.M.:
An Introduction to the Theory of Numbers.
Oxford Univ. Press, Oxford (2008).

\bibitem{KleidmanJ4}
Kleidman, P. B., Wilson, R. A.:
The maximal subgroups of $J_{4}$.
\newblock{\em Proc. London Math. Soc. (3)} 56, no.~3, 484--510. (1988)

\bibitem{KleidmanFi23}
Kleidman, P. B., Parker, R. A., Wilson, R. A.:
The maximal subgroups of the Fischer group $\mathrm{Fi}_{23}$.
\newblock{\em J. London Math. Soc.} 2, no. 1, 89--101. (1989)


\bibitem{KurzweilStellmacher}
Kurzweil, H., Stellmacher, B.:
The Theory of Finite Groups. An Introduction.
Universitext, Springer, New York (2004).

\bibitem{Li1988}
Li, S. R.:
Finite unsolvable groups whose every 2-maximal $3d$-subgroup
is super solvable.
{\em Chinese Ann. Math. Ser. A} 9, 32--37 (1988).
(In Chinese.)

\bibitem{LiMeng}
Li, S.S., Meng, W.: Classification of finite groups satisfying a minimal condition.
{\em Siberian Math. J.} { 50}, no. 1, 100--106 (2009).

\bibitem{LintonTh}
Linton, S. A.:
The maximal subgroups of the Thompson group.
\newblock{\em J. London Math. Soc. (2)} 39 (1989), 79--88.

\bibitem{LintonThCorr}
Linton, S. A.:
Corrections to ``The maximal subgroups of the Thompson group''.
\newblock{\em J. London Math. Soc. 2,} no. 2, 253--254. (1991)

\bibitem{LintonFi24}
Linton, S. A., Wilson, R. A.:
The maximal subgroups of the Fischer groups $\mathrm{Fi}_{24}$ and $\mathrm{Fi}_{24}'$.
\newblock{\em Proc. London Math. Soc. } s3-63, no. 1, 113--164. (1991)

\bibitem{Meng}
Meng, W.,  Lu, J.K.: Finite groups with supersolvable subgroups of even order. {\em Ric. Mat.} { 73}, no. 2, 1059-1064 (2024). DOI: 10.1007/s11587-021-00656-3

\bibitem{MillerMoreno}
Miller, G.A., Moreno, H.C.: Non-abelian groups in which every subgroup is abelian.
{\em Trans. Amer. Math. Soc.} { 4}, no. 4, 398--404 (1903).

\bibitem{Moreto}
Moret\'o, A.:
Finite groups whose maximal subgroups of order divisible by all the primes are supersolvable.
{\em Monatsh. Math.} { 195}, 497--500 (2021).
DOI: 10.1007/s00605-020-01492-7

%\bibitem{Robinson}
% Robinson, D.J.:
% A course in the theory of groups, 2nd ed., Springer-Verlag, New
%York-Heidelberg-Berlin (1996).

\bibitem{Schmidt}
Schmidt, O.Yu.: Groups whose subgroups are all special.
{\em Mat. Sb.} { 31}, no. 3-4, 366--372 (1924).

\bibitem{ShiTian}
Shi, J., Tian, Y.:
On finite groups in which every maximal subgroup of order divisible by $p$ is nilpotent (or abelian).
{\em Rend. Sem. Mat. Univ. Padova} { 154}, 143--149 (2025).
DOI: 10.4171/RSMUP/159

\bibitem{Wilde}
Wilde, T.:
On a recent result of Beltr\'an and Shao.
{\em Results Math.} { 81}, 75 (2026).
DOI: 10.1007/s00025-026-02637-0

\bibitem{WilsonB}
Wilson, R. A.:
The maximal subgroups of the Baby Monster, I.
\newblock{\em J. Algebra} { 211}, no.~1, 1--14.  (1999). 



\end{thebibliography}

\end{document}